\documentclass[a4paper]{amsart}
\usepackage{epsfig}
\usepackage{url}
\usepackage{setspace}
\usepackage{color}

\usepackage[english]{babel}
\usepackage[utf8x]{inputenc}
\usepackage[T1]{fontenc}

\usepackage{bm}
\usepackage{bbold}
\usepackage{amsmath, amsthm, amssymb,  amsfonts, mathrsfs}
\usepackage{mathtools}

\usepackage[a4paper,top=2cm,bottom=2cm,left=2cm,right=2cm,marginparwidth=1.75cm]{geometry}

\newtheorem{theorem}{Theorem}[section]
\newtheorem{cor}[theorem]{Corollary}
\newtheorem{lem}{Lemma}

\newtheorem*{rem}{Remark}

\newcommand{\Mod}[1]{\ (\mathrm{mod}\ #1)}

\newcommand{\lcm}{{\rm lcm}\hspace{0.05cm}}

\title{Large gaps between values of several binary quadratic forms}

\author{B\l{}a\.zej \.Zmija}
\address{Charles University, Faculty of Mathematics and Physics, Department of Mathematical Analysis, and Department of Algebra, Sokolov\-sk\' a 83, 18600 Praha~8, Czech Republic \newline
and Institute of Mathematics of the Polish Academy of Sciences, \'{S}niadeckich 8, 00-656 Warsaw, Poland}
\email{blazej.zmija@gmail.com}

\thanks{The author was supported by Czech Science Foundation grant No. 21-00420M, Charles University Research Centre program UNCE/SCI/022, and PRIMUS Research Programme PRIMUS/25/SCI/017. 
}

\allowdisplaybreaks

\begin{document}

\begin{abstract}
In this paper we study the problem of long gaps between values of binary quadratic forms. Let $D_{1}$, $D_{2},\ldots ,D_{r}$ be negative integers and $(s_{n})_{n=1}^{\infty}$ be the sequence of all the numbers representable by any binary quadratic form of discriminant $D_{1}$, $D_{2}$, $\ldots$ or $D_{r}$, and let $d :=\lcm\{D_{1},\ldots ,D_{r}\}$. We show that then
\begin{align*}
\limsup_{n\to\infty}\frac{s_{n+1}-s_{n}}{\log s_{n}}\geq \frac{1}{\log d + \log\log d + \log\log\log d + 4}.
\end{align*}
This improves and generalises a result by Dietmann, Elsholtz, Kalmynin, Konyagin, and Maynard.

As a by-product of our preliminary results, we show an improvement to the Pólya-Vinogradov inequality.
\end{abstract}

\maketitle

\section{Introduction}

The problem of long gaps between values of binary quadratic forms (usually in the case of sums of two squares) attracted the attention of many mathematicians, including Erd\H{o}s \cite{Erd}, Bambah and Chowla \cite{BC}, Hooley \cite{Hoo}, and Maynard \cite{May}. For many years the best result in this direction was the following. If $\mathcal{S}=(s_{n})_{n=1}^{\infty}$ is the increasing sequence of all the positive integers representable by any binary quadratic form $Q(x,y)=ax^{2}+bxy+cy^{2}$ of discriminant $D:=b^{2}-4ac<0$ then
\begin{align*}
\limsup_{n\to\infty} \frac{s_{n+1}-s_{n}}{\log s_{n}}\geq \frac{1}{|D|}.
\end{align*}
It comes from very short and elegant paper by Richards \cite{Richards}. Very recently, Dietmann, Elsholtz, Kalmynin, Konyagin and Maynard \cite[Theorem 2]{DEKKM} improved on Richards' work and showed that in the special case of sums of two squares we have
\begin{align*}
\limsup_{n\to\infty} \frac{s_{n+1}-s_{n}}{\log s_{n}}\geq \frac{390}{449},
\end{align*}
and in general:
\begin{align*}
\limsup_{n\to\infty} \frac{s_{n+1}-s_{n}}{\log s_{n}}\geq C_{D},
\end{align*}
where
\begin{align}\label{EquCD}
C_{D}=\frac{|D|-1}{2|D| \left(\log\varphi (|D|) +1\right)} \hspace{1cm} \textrm{ or } \hspace{1cm} C_{D}=\frac{|D|}{2\varphi(|D|)\left(\log |D| + O((\log\log |D|)^{2})\right)}.
\end{align}
Here $\varphi(n)$ denotes Euler's totient function. However, there was a typo in their computations. In the proof of \cite[Theorem 2]{DEKKM} (more precisely, on the bottom of page 11334 and the top of page 11335) they bound a certain quantity $\alpha$ from above by an expression containing a quantity $R(|D|)$. They claim that $R(|D|)=\sum_{i=1}^{t}\frac{1}{l_{i}}$ (we define $t$ and $l_{i}$ later in the proof of our Theorem \ref{ThmGaps}), whereas from their reasoning it would only follow that the same bound for $\alpha$ is true but with a slightly worse expression for $R(|D|)$:
\begin{align*}
R(|D|)=1+\frac{l_{t}}{t}\sum_{i=1}^{t-1}i\left(\frac{1}{l_{i}}-\frac{1}{l_{i+1}}\right) = \frac{|D|-1}{\varphi (|D|)}\sum_{i=1}^{t}\frac{1}{l_{i}}.
\end{align*}
By following the rest of their argument, one would obtain:
\begin{align}\label{DEKKM1}
\limsup_{n\to\infty} \frac{s_{n+1}-s_{n}}{\log s_{n}}\geq \frac{\varphi(|D|)}{2|D| \left(\log\varphi (|D|) +1\right)},
\end{align}
and 
\begin{align}\label{DEKKM2}
\limsup_{n\to\infty} \frac{s_{n+1}-s_{n}}{\log s_{n}}\geq \frac{|D|}{2(|D|-1)\left(\log |D| + O((\log\log |D|)^{3}\right)} > \frac{1}{2\log |D| + O((\log\log |D|)^{3})}.
\end{align}
These inequalities are worse than the ones with quantities \eqref{EquCD} by the factor $\frac{|D|-1}{\varphi (|D|)}$.

The problem of long gaps between values of binary quadratic forms received a lot of attention outside number theory when Mallet-Paret and Sell in the paper \cite{MS} used Richards' ideas from \cite{Richards} and applied them to study differential equations. They showed that the gaps in the sequence of positive integers representable by any form with discriminant $D_{1}$, $D_{2},\ldots$ or $D_{r}$ can be arbitrary large (an explicit estimate was later obtained by Gal and Guo in \cite{GG}). The results by Richards \cite{Richards}, and Dietmann \emph{et al.} \cite{DEKKM} deal only with the case of $r=1$.

Here we compile the ideas from \cite{MS} and \cite{DEKKM} to prove the following theorem.

\begin{theorem}\label{ThmGaps}
Let $\mathcal{D}$ be a finite nonempty set of integers with the property that $\prod_{D\in\mathcal{A}}D$ is not a perfect square whenever $\mathcal{A}\subseteq\mathcal{D}$ has odd cardinality. Let $\mathcal{S}=(s_{1}, s_{2},\ldots )$ be the sequence of positive integers, in increasing order, representable by any binary quadratic form of any discriminant $D\in\mathcal{D}$. Let $d:=\lcm\{\ |D| \ |\ D\in\mathcal{D}\ \}$. Then
\begin{align*}
\limsup_{n\to \infty}\frac{s_{n+1}-s_{n}}{\log s_{n}} \geq \frac{1}{\log d + \log\log d + \log\log\log d + 4}.
\end{align*}
\end{theorem}

\begin{rem}
Suppose that we consider gaps between the values attained by a finite number of positive definite binary quadratic forms and $\mathcal{D}$ is the set of their discriminants. Then every $D\in\mathcal{D}$ is negative, so the property that $\prod_{D\in\mathcal{A}}D$ is not a perfect square for $\mathcal{A}\subseteq\mathcal{D}$ of odd cardinality is satisfied because all such products are negative.
\end{rem}

It is worth to noting that our result is fully explicit, improves the inequalities \eqref{DEKKM1} and \eqref{DEKKM2} by a factor of $2$ and gives  smaller error term in the denominator. 

In our proof, we obtain a precise description of the sum of reciprocals of numbers coprime to a given number. We show in Theorem \ref{MAINTHM} that for any fixed $N\geq 2$ and $x>1$,
\begin{align*}
\sum_{\substack{1\leq k\leq x \\ (k,N)=1 }}\frac{1}{k} = \frac{\varphi (N)}{N}\left(\log x + \gamma + \sum_{p\mid N}\frac{\log p}{p-1}\right) + E_{N}(x),
\end{align*}  
where $|E_{N}(x)|\leq \frac{\sigma_{1}^{*}(N)}{8x^{2}}$ and $\sigma_{1}^{*}(N)$ the sum of square-free divisors of $N$. This sum appears (with a different error term) in \cite[2.1.1 Exercise 16(b)]{MV}. Interestingly, we have found one more application of it, which is a refinement of the well-known Pólya-Vinogradov inequality.

\begin{theorem}\label{ThmPolVino}
Let $k\geq 5$ and $\chi$ be a nonprincipal character modulo $q$. Then for every $x$ we have
\begin{align*}
\left|\sum_{m\leq x}\chi (m)\right|\leq \frac{\varphi(q)}{\sqrt{q}}\big(\log q + \log\log q + \log\log\log q + 4\big) .
\end{align*}
\end{theorem}

The original Pólya-Vinogradov inequality gives the bound $<\sqrt{q}\log q$, so we have improved it by the factor $\frac{\varphi(q)}{q}=\prod_{p\mid q}\left(1-\frac{1}{q}\right)$. Although this improvement seems not to be complicated (and might be known to the experts), we could not find it in the literature. For that reason, let us note some corollaries from Theorem \ref{ThmPolVino}.

During the last years, a lot of work was put into proving explicit estimates of the form
\begin{align}\label{PolVinIneqc}
\left|\sum_{m\leq x}\chi (m)\right|\leq (c+o(1))\sqrt{q}\log q,
\end{align}
with the $o(1)$ term given explicitly, see for example papers by Pomerance \cite{Pom}, and Frolenkov and Soundararajan \cite{FS}. According to our best knowledge, the current record belongs to Bordignon \cite{B}, and Bordignon and Kerr \cite{BK} who proved \eqref{PolVinIneqc} with the following constant $c$:
\begin{align}\label{Equc}
c = \left\{\begin{array}{ll}
3/(8\pi), & \textrm{if } \chi (-1)=-1, \\
3/(4\pi^{2}), & \textrm{if } \chi (-1)=1, \\
1/(4\pi), & \textrm{if } \chi (-1)=-1\textrm{ and } q \textrm{ is squarefree}, \\
1/(2\pi^{2}), & \textrm{if } \chi (-1)=1\textrm{ and } q \textrm{ is squarefree}.
\end{array}\right.
\end{align}

Theorem \ref{ThmPolVino} can be used to prove some results in this direction. More precisely, we will show that \eqref{PolVinIneqc} holds with arbitrarily small constant $c$ if only we restrict to numbers $q$ that have many small prime factors.

It is also believed that in fact for every $q$ the property $\left|\sum_{m\leq x}\chi (m)\right|=o(\sqrt{q}\log q)$ holds for all nonprincipal characters $\chi$ modulo $q$. Our Corollary \ref{PolVinIneqCor}(2) below implies that this is true for infinitely many numbers $q$. It is worth emphasising that the result works without any further assumption on $\chi$. 

Recall two facts that will be useful. Mertens' theorem \cite{Mert} says that
\begin{align}\label{MertensThm}
\prod_{p\leq x}\left(1-\frac{1}{p}\right) \sim \frac{1}{e^{\gamma}\log x},
\end{align}
and Nicolas \cite{Nico} showed that inequality
\begin{align}\label{NicolasThm}
\varphi (q) < \frac{q}{e^{\gamma} \log\log q}
\end{align}
holds for infinitely many numbers $q$, where $\gamma$ is the Euler constant. 

\begin{cor}\label{PolVinIneqCor}
\begin{enumerate}
\item For every constant $c>0$ there exists a set $\mathcal{N}_{c}\subseteq\mathbb{N}$ of positive natural density such that for all $q\in \mathcal{N}_{c}$ and all nonprincipal characters $\chi$ modulo $q$ we have
\begin{align*}
\left|\sum_{m\leq x}\chi (m)\right| < c\sqrt{q}\big(\log q + \log\log q + \log\log\log q + 4\big) =(c+o(1))\sqrt{q}\log q.
\end{align*}
\item There are infinitely many numbers $q$ such that for every nonprincipal character $\chi$ modulo $q$:
\begin{align*}
\left|\sum_{m\leq x}\chi (m)\right| < \frac{\sqrt{q}}{e^{\gamma}\log\log q}\big(\log q + \log\log q + \log\log\log q + 4\big) = (e^{-\gamma}+o(1))\frac{\sqrt{q}\log q}{\log\log q}.
\end{align*}
\end{enumerate}
\end{cor}
\begin{proof}
\begin{enumerate}
\item Fix $c>0$. From \eqref{MertensThm} it follows that there exists $x_{c}$ such that $\prod_{p\leq x_{c}}\left(1-\frac{1}{p}\right) < c$. Then it is enough to define $\mathcal{N}_{c}$ to be the set of all the positive integers divisible by $\prod_{p\leq x_{c}}p$. The result follows because of Theorem \ref{ThmPolVino}.
\item It is an easy consequence of \eqref{NicolasThm} and Theorem \ref{ThmPolVino}.
\end{enumerate}
\end{proof}

In particular, we can explicitly write down conditions under which the constants $c$ from \eqref{Equc} are improved. Let $p_{n}$ denote the $n$th prime number and
\begin{align*}
P_{n}:=\prod_{j=1}^{n} p_{j}, \hspace{0.7cm} c_{n}:=\prod_{j=1}^{n}\left(1-\frac{1}{p_{j}}\right).
\end{align*}
Regarding the first part of Corollary \ref{PolVinIneqCor}, one can check using Wolfram Mathematica \cite{Wol} that
\begin{align*}
c_{26} < \frac{3}{8\pi}, \hspace{0.5cm} c_{249}<\frac{3}{4\pi^{2}}, \hspace{0.5cm} c_{187}<\frac{1}{4\pi}, \hspace{0.5cm} c_{6482}<\frac{1}{2\pi^{2}}.
\end{align*}
Therefore, we get that \eqref{PolVinIneqc} holds with the constant $c$ given by:
\begin{align*}
c := \left\{\begin{array}{ll}
c_{26}\approx 0.11912603, & \textrm{if }\ P_{26}\mid q, \\
c_{249}\approx 0.07594335, & \textrm{if }\ P_{249}\mid q, \\
c_{187}\approx 0.07952585, & \textrm{if }\ P_{187}\mid q, \\
c_{6482}\approx 0.05065986, & \textrm{if }\ P_{6482}\mid q,
\end{array}\right.
\end{align*}
and this improves Bordignon and Kerr's results \cite{B,BK} for numbers $q$ that have many small prime divisors.

\section*{Acknowledgements}
I am very grateful to Christian Elsholtz for sharing me with the problem and for his help during the preparation of the paper. I would like to express my gratitude also to Piotr Miska and Bartosz Sobolewski for careful reading of the manuscript and for many suggestions.

\section{Sum of reciprocals of numbers coprime to a given number}

Let us denote for a positive integer $N$ by $\sigma_{1}^{*}(N)$ the sum of square-free divisors of $N$,
\begin{align*}
\sigma_{1}^{*}(N) = \sum_{d\mid N}|\mu (d)|d = \prod_{p\mid N}\left(p+1\right).
\end{align*}

In this section we will prove the following theorem.

\begin{theorem}\label{MAINTHM}
Let $N\geq 2$ be fixed and $x>1$. Then
\begin{align*}
\sum_{\substack{1\leq k\leq x \\ (k,N)=1 }}\frac{1}{k} = \frac{\varphi (N)}{N}\left(\log x + \gamma + \sum_{p\mid N}\frac{\log p}{p-1}\right) + E_{N}(x),
\end{align*}
where $|E_{N}(x)|\leq \frac{\sigma_{1}^{*}(N)}{8x^{2}}$.
\end{theorem}

In the proofs of Theorems \ref{ThmGaps} and \ref{ThmPolVino} we will actually need explicit bounds for the above sum only in the case of $x=N$. We provide these in the following corollary, which is a consequence of Theorem \ref{MAINTHM}.

\begin{cor}\label{MAINCOR}
If $N\geq 5$ then
\begin{align*}
\frac{\varphi (N)}{N}\log N < \sum_{\substack{1\leq k\leq N \\ (k,N)=1 }}\frac{1}{k} < \frac{\varphi (N)}{N}\big(\log N + \log\log N + \log\log\log N + 4\big).
\end{align*}
\end{cor}

In the proofs of Theorem \ref{MAINTHM} and Corollary \ref{MAINCOR} we will need the following lemma.

\begin{lem}\label{HarmonicBound}
For every $n$:
\begin{align*}
\sum_{k=1}^{n}\frac{1}{k}=\log n+\gamma+\frac{1}{2n}+R(n),
\end{align*}
where $-\frac{1}{8n^{2}}\leq R(n)\leq 0$.
\end{lem}
\begin{proof}
See \cite[Section 2]{BW}.
\end{proof}

\begin{proof}[Proof of Theorem \ref{MAINTHM}]
We have by Lemma \ref{HarmonicBound}:
\begin{align*}
\sum_{\substack{1\leq k\leq x \\ (k,N)=1}} \frac{1}{k} = &\ \sum_{k\leq x}\ \sum_{d\mid (k,N)}\frac{\mu(d)}{k} = \sum_{d\mid N}\ \sum_{k\leq \frac{x}{d}}\frac{\mu(d)}{dk}=\sum_{d\mid N}\frac{\mu(d)}{d}\sum_{k\leq \frac{x}{d}}\frac{1}{k} \\
= &\ \sum_{d\mid N}\frac{\mu(d)}{d}\left(\log\left(\frac{x}{d}\right) +\gamma +\frac{d}{2x} +R\left(\frac{x}{d}\right)\right) \\
= &\ \left(\sum_{d\mid N}\frac{\mu (d)}{d}\right)\log x - \sum_{d\mid N}\frac{\mu (d)\log d}{d} + \gamma \sum_{d\mid N}\frac{\mu (d)}{d} +\frac{1}{2x}\sum_{d\mid N}\mu (d) + \sum_{d\mid N}\frac{\mu (d)}{d}R\left(\frac{x}{d}\right) \\
= &\ \frac{\varphi (N)}{N}\log x - \sum_{d\mid N}\frac{\mu (d)\log d}{d} +\gamma\frac{\varphi (N)}{N} + \sum_{d\mid N}\frac{\mu (d)}{d}R\left(\frac{x}{d}\right).
\end{align*}
It is well known (see \cite[2.1.1 Exercise 16(a)]{MV}) that
\begin{align*}
-\sum_{d\mid N}\frac{\mu (d)\log d}{d} = \frac{\varphi (N)}{N}\sum_{p\mid N}\frac{\log p}{p-1}.
\end{align*}
In order to finish the proof, it remains to bound the last sum appearing in the previous expression. By Lemma \ref{HarmonicBound} we have:
\begin{align*}
\left|\sum_{d\mid N}\frac{\mu (d)}{d}R\left(\frac{x}{d}\right)\right| \leq \sum_{d\mid N}\frac{|\mu (d)|}{d}\left|R\left(\frac{x}{d}\right)\right| \leq \frac{1}{8}\sum_{d\mid N}\frac{|\mu (d)|}{d} \cdot \frac{d^{2}}{x^{2}}=\frac{1}{8x^{2}}\sum_{d\mid N}|\mu (d)|d = \frac{\sigma_{1}^{*}(N)}{8x^{2}}.
\end{align*}
The proof is finished.
\end{proof}

For every positive integer $k$ let $p_{k}$ denote the $k$th prime. In order to prove Corollary \ref{MAINCOR} we will need the following bound for $p_{k}$.

\begin{lem}\label{PrimeBound}
For every $k\geq 1$ we have the bound
\begin{align*}
p_{k}\leq 2k(\log k+1).
\end{align*}
\end{lem}
\begin{proof}
If $k\geq 6$ then \cite[Theorems 28 and 30B]{Rosser} and \cite[Section 4]{Dusart} give
\begin{align*}
p_{k}<k(\log k+\log\log k)<2k\log k.
\end{align*}
One can easily check by hand that $p_{k}<2k\log k$ for $k\in\{3,4,5\}$ and $p_{k}\leq 2k(\log k+1)$ for $k\in\{1,2\}$.
\end{proof}

\begin{proof}[Proof of Corollary \ref{MAINCOR}]
At first, let us note that the expression $\frac{\varphi (N)}{N}\sum_{p\mid N}\frac{\log p}{p-1}$ is clearly positive. The lower bound from the statement follows, since if $N\geq 2$, then
\begin{align}\label{bound2/3}
\frac{\varphi(N)/N}{\sigma_{1}^{*}(N)/N^{2}}=\frac{N^{2}\cdot \frac{\varphi (N)}{N}}{\sigma_{1}^{*}(N)} \geq \frac{\left({\rm rad} (N) \right)^{2} \cdot \frac{\varphi(N)}{N}}{\sigma_{1}^{*}(N)} = \prod_{p\mid N} \frac{p^{2}\left(1-\frac{1}{p}\right)}{p+1} \geq \frac{2}{3}\prod_{p\mid N,\ p>2} \frac{p(p-1)}{p+1}\geq \frac{2}{3},
\end{align}
and therefore 
\begin{align*}
\gamma \frac{\varphi (N)}{N} - \frac{\sigma_{1}^{*}(N)}{8N^{2}} \geq \left(\gamma - \frac{3}{16}\right)\frac{\varphi (N)}{N} > \frac{\varphi (N)}{3N} > 0.
\end{align*}

For the upper bound let $h$ be the number of distinct prime divisors of $N$. We have:
\begin{align*}
\sum_{p\mid N}\frac{\log p}{p-1} & \leq \sum_{k=1}^{h}\frac{\log p_{k}}{p_{k}-1} =\sum_{k=1}^{h}\frac{\log p_{k}}{p_{k}} + \sum_{k=1}^{h}\frac{\log p_{k}}{p_{k}(p_{k}-1)} \\ 
& < \sum_{k=1}^{h}\frac{\log p_{k}}{p_{k}} + \sum_{k=1}^{h}\frac{\log p_{k}}{p_{k}(p_{k}-1)} = \sum_{k=1}^{h}\frac{\log p_{k}}{p_{k}} -E-\gamma ,
\end{align*}
where the constant $E$ is approximately computed in \cite[(2.8) and (2.11)]{RS}, $E\in (-1.33259,-1.33258)$. 

In order to estimate the sum $\sum_{k=1}^{h}\frac{\log p_{k}}{p_{k}}$, we directly apply inequality \cite[(3.24)]{RS} and obtain:
\begin{align*}
\sum_{p\mid N}\frac{\log p}{p-1} & < \log p_{h} - E - \gamma.
\end{align*}

From Lemma \ref{PrimeBound} and the trivial bound $h\leq \frac{\log N}{\log 2}$ we get:
\begin{align*}
\log p_{h} & \leq  \log h+\log(\log h +1) + \log 2 \\
& \leq \log\log N - \log\log 2 + \log\left(\log\log N - \log\log 2 +1\right) + \log 2 \\
& = \log\log N + \log\log\log N + \log\left(1 + \frac{1-\log\log 2}{\log\log N}\right) + \log 2 - \log\log 2 \\
& \leq \log\log N + \log\log\log N + \log\left(1 + \frac{1-\log\log 2}{\log\log 5}\right) + \log 2 - \log\log 2.
\end{align*}
Therefore,
\begin{align*}
\sum_{p\mid N}\frac{\log p}{p-1} & \leq \log\log N+\log\log\log N + \log\left(1 + \frac{1-\log\log 2}{\log\log 5}\right) + \log 2 - \log\log 2-E-\gamma \\
& < \log\log N+\log\log\log N + 3.75 -\gamma.
\end{align*}
This leads to the final bound 
\begin{align*}
\sum_{\substack{1\leq k\leq N \\ (k,N)=1}} \frac{1}{k} \leq &\ \frac{\varphi (N)}{N}\big(\log N+\log\log N +\log\log\log N +3.75\big) +\frac{\sigma_{1}^{*}(N)}{8N^{2}} \\
 \leq  &\  \frac{\varphi (N)}{N}\left(\log N+\log\log N +\log\log\log N + \frac{23}{6} \right) \\
 < &\ \frac{\varphi (N)}{N}\big(\log N+\log\log N +\log\log\log N + 4 \big)
\end{align*}
because of \eqref{bound2/3}.
\end{proof}

\section{Large gaps between values of binary quadratic forms}

Recall that for an integral binary quadratic form $Q(x,y)=ax^{2}+bxy+cy^{2}$ we define its discriminant as $D:=b^{2}-4ab$. We say that an integer $n$ is represented by $Q$ is $Q(x,y)=n$ for some integer $x$ and $y$. For a prime number $p$ and positive integers $n$ and $\alpha$ we write $p^{\alpha}\| n$ if $\alpha$ is the highest power of $\alpha$ that divides $n$. The symbol $\left(\frac{m}{n}\right)$ denotes the Kronecker symbol $m$ over $n$.

Recall some auxiliary results from the paper \cite{MS}.

\begin{lem}\label{LemNonRes}
Let $\mathcal{D}$ be a finite nonempty set of integers with the property that $\prod_{D\in\mathcal{A}}D$ is not a perfect square whenever $\mathcal{A}\subseteq\mathcal{D}$ has odd cardinality. Then there exists an integer $r\neq 0$ such that
\begin{align*}
(D,r)=1 \ \ \ \ \ \textrm{ and } \ \ \ \ \ \left(\frac{D}{r}\right) = -1
\end{align*}
for each $D\in\mathcal{D}$.
\end{lem}
\begin{proof}
See \cite[Lemma A.1]{MS}.
\end{proof}

\begin{lem}\label{LemPrime}
Let $Q$ be a quadratic form with discriminant $D$. If $p$ is a prime and $\alpha$ is an odd number such that $\left(\frac{D}{p}\right)=-1$ and $p^{\alpha}\| n$, then $n$ is not represented by $Q$.
\end{lem}
\begin{proof}
See \cite[Lemma A.2]{MS} or \cite[Lemma 7]{DEKKM}.
\end{proof}

\begin{proof}[Proof of Theorem \ref{ThmGaps}] We follow Richards' original idea from \cite{Richards} with the improvements from \cite{MS} and \cite{DEKKM}. Let
\begin{align*}
d:=\lcm\big\{\ |D|\ \big|\ D\in\mathcal{D}\ \big\}
\end{align*}
and $r$ be the the number from Lemma \ref{LemNonRes}. We can assume that $1\leq r\leq d$.

Let $t=\varphi (d)$ and for $j\in\{1,\ldots ,t\}$ define $\ell_{j}$ to be the $j$th residue class (in increasing order) modulo $d$ and coprime to $d$. In particular, $\ell_{1}=1$ and $\ell_{t}=d-1$. Let
\begin{align*}
T_{i}:=\big\{\ x\in (\mathbb{Z}/d\mathbb{Z})^{*}\ |\ l_{j}x\equiv r\ \Mod{d} \textrm{ for some } j\leq i \ \big\}.
\end{align*}
Let $\varpi$ be the projection, $\varpi :\mathbb{Z}\to\mathbb{Z}/d\mathbb{Z}$.

Let us fix $\varepsilon >0$ and take a positive integer $k$ (the size of the gaps) sufficiently large in terms of $\varepsilon$ and $d$. Define
\begin{align*}
L:=\max\big\{\ |r+dj| \ \big|\ j=0,\ldots ,k \ \big\}=r+dk\leq d(k+1).
\end{align*}
For every $p\leq L$, $p\nmid d$, let $\beta_{p}:=\left\lfloor\frac{\log L}{\log p}\right\rfloor$. That is, $p^{\beta_{p}}\leq L<p^{\beta_{p}+1}$. 

Now define
\begin{align*}
P:=\prod_{\substack{ p_{t}\leq L/\ell_{t} \\ (p_{t},d)=1 }}p_{t}^{\beta_{p_{t}}+1} \ \prod_{i=1}^{t-1} \ \prod_{\substack{ L/\ell_{i+1} < p_{i}\leq L/\ell_{i} \\ \varpi (p_{i})\in T_{i} }}p_{i}^{\beta_{p_{i}}+1} ,
\end{align*}
where the products are taken over prime numbers $p_{1},\ldots ,p_{t}$.

Finally, let $y\in\{1,\ldots ,P\}$ be such that 
\begin{align*}
dy\equiv r\pmod{P} .
\end{align*}
One can follow the argument from the end of the proof of \cite[Theorem 2]{DEKKM} to show that all the numbers $y+1$, $y+2$, \ldots , $y+k$ are not represented by forms of discriminant from the set $\mathcal{D}$. This would lead to bounds similar to \eqref{DEKKM1} and \eqref{DEKKM2}. Instead of that, we can use the idea from \cite{DEKKM,KK} in order to gain an additional factor $2$. Let $\delta\in (0,1/l_{t})$ and
\begin{align}\label{EqPdelta}
P_{\delta}:=\prod_{\substack{p\mid P \\ p\leq\delta L}}p^{\beta_{p}+1}\ \prod_{\substack{p\mid P \\ p>\delta L}}p.
\end{align}
Now, define $y_{0}\in\{1,\ldots ,P_{\delta}\}$ to be the number satisfying $y_{0}\equiv y\pmod{P_{\delta}}$.

For every $1\leq n\leq \delta L$ consider the interval $I_{n}:=[1+y_{0}+nP_{\delta},k+y_{0}+nP_{\delta}]$.

\begin{lem}
For some $1\leq n \leq \delta L$ we have $\mathcal{S}\cap I_{n}=\emptyset$.
\end{lem}
\begin{proof}
Suppose that $m\in\mathcal{S}\cap I_{n}$ for some $n$. Then there exists $j\in\{1,\ldots,k\}$ such that $m=j+y_{0}+nP_{\delta}$. Hence,
\begin{align*}
dm\equiv d(j+y_{0})\equiv dj+dy\equiv dj+r \pmod{P_{\delta}}.
\end{align*}
However,
\begin{align*}
-1=\left(\frac{D}{r}\right)=\left(\frac{D}{dj+r}\right)
\end{align*}
for every $D\in\mathcal{D}$. Hence, there exists a prime $p$ and an odd number $\gamma$ such that $\left(\frac{D}{p}\right)=-1$ and $p^{\gamma}\| dj+r$. Then $\gamma\leq\beta_{p}$, or otherwise $p^{\gamma}>L\geq dj+r$. Write $dj+r=p^{\gamma}l$. Hence, if $\gamma\geq 3$ then $p^{3}\leq L$, so both $p\leq L/l_{t}$ and $p\leq\delta L$ are true if $L$ is big enough. Therefore, we get $p^{\beta_{p}+1}\mid P_{\delta}$. Hence, $p^{\gamma}\| m$, so $m\not\in\mathcal{S}$ by Lemma \ref{LemPrime}, a contradiction.

It follows from the previous paragraph that $\gamma =1$. We can also assume that $p>\delta L$ (otherwise $p\| m$, a contradiction by Lemma \ref{LemPrime}). Then $p\mid m$, and since $m\in\mathcal{S}$ and $\left(\frac{D}{p}\right)=-1$, we have $p^{2}\mid m$. Moreover, there exists $i\in\{1,\ldots, t\}$ such that $L/l_{i+1}<p\leq L/l_{i}$, so $l\leq l_{i}$ and hence $\varpi (p)\in T_{i}$, so $p\mid P_{\delta}$.

Now observe that for every pair $(j,p)$, where $j\leq k$ and $p>\delta L$, there is at most one $n<\delta L$ such that $p^{2}\mid j+y_{0}+nP_{\delta}$. Indeed, if it is true for numbers $n_{1}$ and $n_{2}$, then $p^{2}\mid (n_{1}-n_{2})P_{\delta}$, so $p\mid n_{1}-n_{2}$ because $p\| P_{\delta}$. However, $|n_{1}-n_{2}|\leq \delta L<p$, so $n_{1}=n_{2}$.

Moreover, if $p\mid j_{1}+y_{0}$ and $p\mid j_{2}+y_{0}$, then $p\mid j_{1}-j_{2}$. Hence, there are at most $\frac{L}{p}+1<\frac{1}{\delta}<\frac{2}{\delta}$ numbers $j$ such that $p\mid j+y_{0}$.

Let us summarize. So far we have proved that if $\mathcal{S}\cap I_{n}\neq \emptyset$ then there exist $j\leq L$ and a prime $p>\delta L$ such that $p\mid P_{\delta}$ and $p^{2}\mid j+y_{0}+nP_{\delta}$, and for every such a $p$ there are at most $\frac{2}{\delta}$ choices of $j$. Therefore,
\begin{align*}
\#\{n\leq\delta L\ |\ \mathcal{S}\cap I_{n}\neq\emptyset\} &\leq \#\big\{ (j,p)\ \big|\ j\leq L,\ p>\delta L,\ p\mid P_{\delta},\ p\mid j+y_{0}\big\}<\frac{2}{\delta}\#\{ p>\delta L\ |\ p\textrm{ divides } P_{\delta}\ \}.
\end{align*}
Define $F:=\#\{ p>\delta L\ |\ p\textrm{ divides } P_{\delta}\ \}$. Then 
\begin{align*}
(\delta L)^{F}\leq P_{\delta}<L^{2\pi (L)}<L^{2(1+\varepsilon )\frac{L}{\log L}}
\end{align*}
if $L$ is large enough, where $\pi (L)$ denotes the number of primes not greater than $L$. Therefore,
\begin{align*}
\#\{n\leq \delta L\ |\ \mathcal{S}\cap I_{n}\neq\emptyset\}<\frac{4(1+\varepsilon )}{\delta}\cdot\frac{L}{\log (\delta L)}.
\end{align*}
If $L$ is large enough, the last quantity is smaller than $\delta L$. The lemma is proved.
\end{proof}

\vspace{\baselineskip}

From the above lemma we get that there exists $n$ such that $s_{n}<P_{\delta}$ and
\begin{align*}
s_{n+1}-s_{n}\geq k.
\end{align*}
Hence, in order to finish the proof it is enough to find a bound of the form $k\geq C_{d}\log P_{\delta}$ for appropriate $C_{d}$. 

Let $\varepsilon>0$. The product over primes $p\leq \delta L$ can be bounded using the prime number theorem:
\begin{align}\label{IneqProdpbeta+1}
\prod_{p\leq\delta L} p^{\beta_{p}+1}<L^{2\pi (\delta L)}<L^{2(1+\varepsilon )\frac{\delta L}{\log L}} = e^{2(1+\varepsilon )\delta L}
\end{align}
if $L$ is large enough. 

The part containing primes $>\delta L$ is more difficult to estimate. Recall that $t=\varphi (d)$ and for every $i$ the set $T_{i}$ consists of exactly $i$ classes modulo $d$. Put 
\begin{align}\label{Eqlt+1}
l_{t+1}:=1/\delta
\end{align}
and for every $X>0$ and integers $A$, $B\in\{0,\ldots ,B-1\}$ let
\begin{align*}
\pi (X;A,B) := \#\left\{\ p\leq X\ |\ p\equiv A \pmod{B}\ \right\}.
\end{align*}

We get from the prime number theorem for arithmetic progressions and the prime number theorem that for every $i\in\{1,\ldots ,t\}$, since $\# T_{i}= i$, the following is true if $L$ is big enough:
\begin{align*}
\#\left\{\ \frac{L}{l_{i+1}}<p\leq\frac{L}{l_{i}}\ \bigg|\ \varpi (p)\in T_{i} \ \right\} & = \sum_{j\in T_{i}} \left(\pi\left(\frac{L}{l_{i}};j,d\right) - \pi\left(\frac{L}{l_{i+1}};j,d\right) \right) < \sum_{j\in T_{i}}\frac{1+\varepsilon}{t}\left(\pi\left(\frac{L}{l_{i}}\right)-\pi\left(\frac{L}{l_{i+1}}\right)\right) \\
&= (1+\varepsilon) \frac{\# T_{i}}{t}\left(\pi\left(\frac{L}{l_{i}}\right)-\pi\left(\frac{L}{l_{i+1}}\right)\right) = (1+\varepsilon)\frac{i}{t}\left(\pi\left(\frac{L}{l_{i}}\right)-\pi\left(\frac{L}{l_{i+1}}\right)\right) \\
&< (1+2\varepsilon)\frac{i}{t}\left(\frac{L}{l_{i}\log\left(\frac{L}{l_{i}}\right)}-\frac{L}{l_{i+1}\log\left(\frac{L}{l_{i+1}}\right)}\right).
\end{align*}
Observe further, that for every $i$, if $L$ is big enough:
\begin{align*}
\frac{1}{l_{i}\log\left(\frac{L}{l_{i}}\right)}-\frac{1}{l_{i+1}\log\left(\frac{L}{l_{i+1}}\right)} & = \frac{1}{l_{i}\log L}\cdot \frac{1}{1-\frac{\log l_{i}}{\log L}} - \frac{1}{l_{i+1}\log L}\cdot \frac{1}{1-\frac{\log l_{i+1}}{\log L}} \\
& < \frac{1+3\varepsilon}{1+2\varepsilon}\left(\frac{1}{l_{i}\log L} - \frac{1}{l_{i+1}\log L}\right) \\
& = \frac{1+3\varepsilon}{1+2\varepsilon} \cdot \frac{1}{\log L}\left(\frac{1}{l_{i}}-\frac{1}{l_{i+1}}\right).
\end{align*}
Hence,
\begin{align*}
\#\left\{\ \frac{L}{l_{i+1}}<p\leq\frac{L}{l_{i}}\ \bigg|\ \varpi (p)\in T_{i} \ \right\} < (1+3\varepsilon)\frac{i}{t}\frac{L}{\log L}\left(\frac{1}{l_{i}}-\frac{1}{l_{i+1}}\right).
\end{align*}
Therefore, $\prod\limits_{\substack{p>\delta L \\ p\mid P_{\delta}}}p<L^{\alpha}=e^{\alpha\log L}$, where
\begin{align*}
\alpha &\leq \sum_{i=1}^{t} \#\left\{\ \frac{L}{l_{i+1}}<p\leq\frac{L}{l_{i}}\ \bigg|\ \varpi (p)\in T_{i} \ \right\} < (1+3\varepsilon )\frac{L}{t\log L} \sum_{i=1}^{t}\left(\frac{i}{l_{i}}-\frac{i}{l_{i+1}}\right) \\
&= (1+3\varepsilon)\frac{L}{t\log L}\left(\sum_{i=1}^{t}\frac{1}{l_{i}}-\frac{t}{l_{t+1}}\right) = (1+3\varepsilon)\frac{L}{t\log L}\left(\sum_{i=1}^{t}\frac{1}{l_{i}}-\delta t\right).
\end{align*}
In the last equality, we used equality \eqref{Eqlt+1}.

Using the above inequality together with \eqref{EqPdelta} and \eqref{IneqProdpbeta+1}, we obtain
\begin{align*}
\log P_{\delta} = \log \left(\prod_{\substack{p\mid P \\ p\leq\delta L}}p^{\beta_{p}+1}\right) + \log \left(\prod_{\substack{p\mid P \\ p>\delta L}}p \right) < 2(1+\varepsilon )\delta L + \alpha \log L.
\end{align*}
Further, Corollary \ref{MAINCOR} and inequality $L\leq d(k+1)$ give
\begin{align*}
\log P_{\delta} &< 2(1+\varepsilon )\delta L + (1+3\varepsilon )\frac{L}{t}\sum_{i=1}^{t}\frac{1}{l_{i}}-(1+3\varepsilon )\delta L \\
&= (1+3\varepsilon )\frac{L}{t}\sum_{i=1}^{t}\frac{1}{l_{i}} +(1-\varepsilon)\delta L \\
&< (1+3\varepsilon )\frac{d(k+1)}{\varphi (d)} \sum_{i=1}^{t}\frac{1}{l_{i}} + \delta d (k+1) \\
&< \left((1+3\varepsilon )\frac{d}{\varphi (d)}\frac{\varphi (d)}{d} (\log d+\log\log d+\log\log\log d+ 4) +\delta d\right)(k+1) \\
&< (1+4\varepsilon )(\log d+\log\log d+\log\log\log d+ 4)k,
\end{align*}
where the last inequality is true because we can assume that $\delta$ is small compared to $\varepsilon$, and $k$ is big enough.

From the above consideration we get that for every $\varepsilon_{1} >0$ there are infinitely many numbers $n$ such that
\begin{align*}
\frac{s_{n+1}-s_{n}}{\log s_{n}}>\frac{1-\varepsilon_{1}}{\log d + \log\log d + \log\log\log d + 4}.
\end{align*}
Since $\varepsilon_{1} $ can be arbitrarily small, we get the result after taking the $\limsup$ on both sides of the inequality.
\end{proof}

\section{Pólya-Vinogradov inequality}

\begin{proof}[Proof of Theorem \ref{ThmPolVino}]
We focus on the case of $\chi$ primitive. The general case can be obtained using the idea from the proof of \cite[Theorem 13.15]{Apostol}. We proceed as in the proof of \cite[Theorem 8.21]{Apostol}. By \cite[Theorem 8.20]{Apostol}:
\begin{align*}
\chi (m)=\frac{\tau_{q}(\chi)}{\sqrt{q}}\sum_{k=1}^{q}\bar{\chi}(k)e^{-2\pi imk/q},
\end{align*}
where $\tau_{q}(\chi)$ has absolute value $1$. Thus
\begin{align*}
\sum_{m\leq x}\chi (m)=\frac{\tau_{q}(\chi)}{\sqrt{q}}\sum_{k=1}^{q}\bar{\chi}(k)\sum_{m\leq x}e^{-2\pi imk/q}
\end{align*}
and
\begin{align*}
\sqrt{q}\left|\sum_{m\leq x}\chi (m)\right| \leq \sum_{k=1}^{q}|\chi(k)|\left|\sum_{m\leq x}e^{-2\pi imk/q}\right|.
\end{align*}
Let 
\begin{align*}
f(k)=\sum_{m\leq x}e^{-2\pi imk/q}.
\end{align*}
Then $|f(k)|=|f(q-k)|$ and $(q,k)=(q,q-k)$. Thus,
\begin{align*}
\sqrt{q}\left|\sum_{m\leq x}\chi (m)\right| \leq 2\sum_{1\leq k\leq q/2}|\chi(k)||f(k)|.
\end{align*}

Then it is proved in \cite{Apostol} that $|f(k)|\leq \frac{q}{2k}$. We also have $\chi (k)=0$ if $(k,q)>1$. Therefore we get
\begin{align*}
\sqrt{q}\left|\sum_{m\leq x}\chi (m)\right|\leq q\sum_{1\leq k\leq q/2}\frac{|\chi(k)|}{k}\leq q\sum_{\substack{1\leq k\leq q \\ (k,q)=1}}\frac{1}{k}.
\end{align*}
The latter expression is bounded using Corollary \ref{MAINCOR}:
\begin{align*}
\sqrt{q}\left|\sum_{m\leq x}\chi (m)\right|\leq \varphi (q)\big(\log q + \log\log q + \log\log\log q +4\big).
\end{align*}
The result follows after dividing both sides by $\sqrt{q}$.
\end{proof}


\begin{thebibliography}{abcd}

\bibitem{Apostol} T. Apostol, \emph{ Introduction to Analytic Number Theory}, Springer Science \& Business Media (1998).

\bibitem{BC} R. P. Bambah and S. Chowla, {\it On numbers which can be expressed as a sum of two squares}, Proceedings of the National Institute of Sciences of India, 13 (1947), 101--103.

\bibitem{BW} Ralph P. Boas Jr., John W. Wrench Jr., \emph{Partial sums of the harmonic series}, The American Mathematical Monthly, 78.8 (1971), 864--870.

\bibitem{B} M. Bordignon, {\it Partial Gaussian sums and the Po’lya-Vinogradov inequality for primitive characters}, Revista Matemática Iberoamericana, 38, no. 4 (2022), 1101--1127.

\bibitem{BK} M. Bordignon, B. Kerr, {\it An explicit Pólya-Vinogradov inequality via Partial Gaussian sums}, Transactions of the American Mathematical Society, vol. 373 (2020), 6503--6527.

\bibitem{DEKKM} R. Dietmann, Ch. Elsholtz, A. Kalmynin, S. Konyagin, J. Maynard, \emph{Longer Gaps Between Values of Binary Quadratic Forms}, 2023.12 (2023), 10313--10349.

\bibitem{Dusart} P. Dusart, {\it The $k$th prime is greater than $k(\log k + \log \log k - 1)$ for $k \geq 2$}, Mathematics of Computation, 68 (225) (1999), 411--415.

\bibitem{Erd} P. Erd\H{o}s, {\it Some problems and results in elementary number theory}, Publicationes Mathematicae Debrecen 2 (1951), 103--109.

\bibitem{FS} D. A. Frolenkov, K. Soundararajan, {\it A generalization of the Pólya-Vinogradov inequality}, Ramanujan Journal 31 (2013), no. 3, 271--279.

\bibitem{GG} C. G. Gal, Y. Guo, {\it Inertial manifolds for the hyperviscous Navier–Stokes equations}, Journal of Differential Equations 265.9 (2018), 4335--4374.

\bibitem{Hoo} Ch. Hooley, {\it On the intervals between numbers that are sums of two squares. IV}, Journal für die reine und angewandte Mathematik, 452 (1994), 79--109.

\bibitem{KK} A. Kalmynin, S. Konyagin, \emph{Large gaps between sums of two squares}, arXiv preprint arXiv:1906.09100 (2019).

\bibitem{MS} J. Mallet-Paret, G. R. Sell, \emph{Inertial manifolds for reaction diffusion equations in higher space 
dimensions}, Journal of the American Mathematical Society 1 (1988), 805--866.

\bibitem{May} J. Maynard, {\it Sums of two squares in short intervals}, in: Analytic number theory, pages 253--273. Springer, Cham, 2015.

\bibitem{Mert} F. Mertens, {\it Ein Beitrag zur analytischen Zahlentheorie}, Journal für die reine und angewandte Mathematik 78 (1874), 46--62.

\bibitem{MV} H. L. Montgomery, R. C. Vaughan, \emph{Multiplicative number theory I: Classical theory}, No. 97. Cambridge university press, 2007.

\bibitem{Nico} J.-L. Nicolas, {\it Petites valeurs de la fonction d’Euler}, Journal of Number Theory 17 (1983): 375--388.

\bibitem{Pom} C. Pomerance, {\it Remarks on the Pólya-Vinogradov inequality}, Integers 11 (2011), no 4, 531--542.

\bibitem{Richards} I. Richards, \emph{On the gaps between numbers which are sums of two squares}, Advances in Mathematics 46.1 (1982), 1--2.

\bibitem{Rosser} J. B. Rosser, \emph{Explicit Bounds for Some Functions of Prime Numbers}, American Journal of Mathematics
Vol. 63, No. 1 (Jan., 1941), 211--232.

\bibitem{RS} J. B. Rosser, L. Schoenfeld, \emph{Approximate formulas for some functions of prime numbers}, Illinois Journal of Mathematics 6.1 (1962): 64--94.

\bibitem{Wol} S. Wolfram, The Mathematica Book, 3rd edition, Wolfram Media/Cambridge University Press, Cambridge (2003).

\end{thebibliography}
\end{document}